\documentclass[12pt, a4paper]{article}

\usepackage[english]{babel}
\usepackage[width=478.006296pt, height=676.037472pt]{geometry}
\usepackage{parskip}
\usepackage{biblatex}
\usepackage{amsmath,amssymb,amsthm}
\usepackage{graphicx}

\newtheorem{theorem}{Theorem}[section]

\theoremstyle{definition}
\newtheorem{definition}[theorem]{Definition}
\newtheorem{example}[theorem]{Example}

\theoremstyle{remark}
\newtheorem{remark}[theorem]{Remark}

\newcommand{\N}{\mathbb N}
\newcommand{\ord}{\mathbf{Ord}}
\newcommand{\op}[2]{\langle#1, #2\rangle}

\title{A higher arithmetic on the ordinals}
\author{Adrian Ducourtial}

\date{July 5\textsuperscript{th}, 2026}

\begin{document}

\maketitle

\section{Introduction}

The concept of iteration is central to arithmetic. From addition on the natural numbers $\N = \{0, 1, 2, \dots\}$, one may define multiplication as iterated addition, and exponentiation as iterated multiplication. Pictorally, one may write

\begin{equation}
    a \cdot b = \underbrace{a+\cdots+a}_{b}, \qquad\qquad a^b = \underbrace{a\ \cdot\ \cdots\ \cdot\ a}_{b}.
\end{equation}

The development of arithmetic can be extended further in many ways. In this paper, we focus on two such extensions, and ask whether they can be combined into a common theory.

On the one hand, further operations may be defined. Following Goodstein \cite{Goo1947}, we continue the scheme alluded to in (1) by making it more precise, providing such things as iterated exponentiation, and at length giving an infinite sequence of so-called \emph{hyperoperations}. We define and motivate these in Section 2.

On the other hand, further numbers may be introduced. We consider the \emph{ordinal numbers} of Cantor \cite{Can1883}, serving as a platform for transfinite induction. We assume that the reader is familiar with this number system; in particular, the operations of natural arithmetic above are embedded and generalized in ordinal arithmetic. Inspired by Section 2, we attempt to define higher operations on the ordinals in Section 3. Nonetheless, we learn that the straightforward approaches do not yield satisfactory results.

From this experience, we are inclined to take a different perspective on the situation, leading us to propose a new sequence of higher ordinal operations in Section 4. In this final section, we also show that the traditional arithmetic on the ordinals, as well as the hyperoperations on the naturals, are both embedded in the sequence, thereby providing a common generalization of the two arithmetics.

Looking outward, we are not the first to consider such a sequence. Finsler \cite{Fin1951} and Doner and Tarski \cite{DoTa1969} both study mutually similar sequences of ordinal operations, capturing a form of higher arithmetic iteration. On the finite ordinals, however, neither sequence agrees with the corresponding hyperoperations on the naturals, and the latter sequence does not replicate exponentiation exactly (the former takes it for granted). Indeed, it appears that their research is unrelated to the one of Goodstein, so their motivations are arguably different from ours.

Although the present operations embed both the hyper- and ordinal-arithmetic ones, we have yet to study them in their own right. First, it would be relevant to assess their monotonicity and continuity, especially compared to the usual arithmetic on the ordinals. Second, following \cite{DoTa1969}, one could determine the \emph{main numbers} of our operations, i.e., those ordinals whose initial segments (of all smaller ordinals) are closed under the given operation. In relation to the latter, we conjecture that the $\alpha$\textsuperscript{th} infinite main number of the operation $\op--_n$ for $n \in \omega$ is equal to $\op\omega{\omega^\alpha}_{n+1}$. In any case, both lines of inquiry would clarify the nature and coherency of our unified arithmetic.

\section{Hyperoperations}

As is customary in set theory, we denote by $\omega$ the set of nonnegative integers. Let $* : \omega\times\omega \to \omega$ be a binary operation with right identity element $e \in \omega$. We consider a derived operation $*' : \omega\times\omega \to \omega$, given by

\begin{align*}
    a *' 0\phantom{(b+1)}\kern-.25em &=\ e,\\
    a *' (b+1)\phantom{0}\kern-.25em &=\ a * (a *' b).
\end{align*}

For some $a\in\omega$, we effectively apply the map $a * - : \omega \to \omega$ repeatedly to the right identity element $e$. Importantly, the second argument $b$ is used to \emph{count} the depth of the iteration. For $a\in\omega$, we observe that

\begin{align*}
    a *' 1 &= a * (a *' 0) \\
           &= a * e \\
           &= a,
\end{align*}

and hence, $e' = 1 \in \omega$ is a right identity element of $*'$. Therefore, we may apply the same construction as above and derive another operation $*''$, which, by the same argument, would also have $e'' = 1$ as a right identity element. Therefore, this derivation of higher operations can be continued indefinitely.

More concretely, we now let $* = +$ be the usual addition on the naturals. We have that $0\in\omega$ is the only right identity element of addition, so we may apply the construction scheme above repeatedly. Writing $a+b = [a,b]_0$ as the initial operation, we let successive integers denote the higher operations so obtained, amounting to an infinite sequence of binary operations. We formulate this as a recursive definition.

\begin{definition}\label{hypop}
    For all $n \in \omega$, the operation $[-,-]_n : \omega\times\omega \to \omega$ is given by

    \begin{align*}
        [a, b]_0 &= a+b,\\
        [a, 0]_1 &= 0,\\
        [a, 0]_n &= 1 \quad \text{for } n \geq 2,\\
        [a, b+1]_{n+1} &= [a, [a, b]_{n+1}]_n.
    \end{align*}

\end{definition}

We note that the right identity element of some operation $[-,-]_n$ is determined by the index of the higher operation $[-,-]_{n+1}$. After addition, the next couple of terms are some familiar operations.

\begin{theorem}
    For all $a, b \in\omega$, the following hold.
    \begin{enumerate}
        \item[\textnormal{(i)}] $[a,b]_1 = a\cdot b$
        \item[\textnormal{(ii)}] $[a,b]_2 = a^b$
    \end{enumerate}
\end{theorem}

\begin{proof}
    Let $a\in\omega$ be given. We prove each statement in order, by simple induction on $b$.
    \begin{enumerate}
        \item[(i)] For $b=0$, we have $[a,0]_1 = 0 = a\cdot 0$. For the inductive step, suppose $[a,b]_1 = a\cdot b$ for some $b\in\omega$. We have
        \begin{align*}
            [a, b+1]_1 &= [a, [a,b]_1]_0 && \\
                       &= a + [a,b]_1 && \text{by definition} \\
                       &= a + (a\cdot b) && \text{by inductive hypothesis} \\
                       &= a \cdot (b+1).
        \end{align*}
        \item[(ii)] For $b=0$, we have $[a,0]_2 = 1 = a^0$ (in particular, $0^0 = 1$). For the inductive step, suppose $[a,b]_2 = a^b$ for some $b\in\omega$. We have
        \begin{align*}
            [a, b+1]_2 &= [a, [a,b]_2]_1 && \\
                       &= a \cdot [a,b]_2 && \text{by (i)} \\
                       &= a \cdot a^b && \text{by inductive hypothesis} \\
                       &= a^{b+1}.
        \end{align*}
    \end{enumerate}
\end{proof}

For $n = 3$, the operation $[-,-]_3$ is a form of iterated exponentiation, where we iterate along the second argument (i.e., the exponent). As exponentiation is non-associative, we may define an alternative sequence of operations in terms of iteration along the first argument. This would give

\[[a, b+1]_{n+1} = [[a, b]_{n+1}, a]_n\]

in the general case. Therefore, we expect the resulting operations to have \emph{left} identity elements instead. The first operation, addition, has $0\in\omega$ as a left identity element, and it is possible to recuperate multiplication and exponentiation by this route (the former has 1 as a left identity element). Continuing further in this manner, however, is not possible.

\begin{theorem}
    Exponentiation has no left identity element.
\end{theorem}

\begin{proof}
    Suppose, aiming at a contradiction, that there exists an $e\in\omega$ such that $e^a = a$ for all $a\in\omega$. Then, we have $e^2 = 2$, but 2 is not a perfect square.
\end{proof}

Thus, we are compelled to use ``rightward'' iteration. This is a stronger notion of iteration, since it applies to the depth of the lower iteration itself (the argument $b$).

\section{Ordinals}

We denote the class of ordinals by $\ord$, and its usual well-order by $<$. Since the arithmetic on the ordinals embeds the arithmetic on the naturals, we denote the class functions of the former identically to the operations of the latter (and simply call both operations). Taking ordinal addition for granted, we recall the definitions of ordinal multiplication and exponentiation by transfinite recursion. First, for $\gamma > 0$, the \emph{limit supremum} and \emph{limit infimum} of a $\gamma$-sequence $(\alpha_\xi)_{\xi < \gamma}$ are defined as
\begin{align*}
    \textstyle\limsup_{\xi < \gamma} \alpha_\xi &= \inf {\{\sup {\{\alpha_\xi : \delta \leq \xi < \gamma\}} : \delta < \gamma\}}\\
    \textstyle\liminf_{\xi < \gamma} \alpha_\xi &= \sup {\{\inf {\{\alpha_\xi : \delta \leq \xi < \gamma\}} : \delta < \gamma\}},
\end{align*}
respectively. If these two values coincide, we simply talk of the \emph{limit} of the $\gamma$-sequence, which we denote by $\lim_{\xi < \gamma} \alpha_\xi$.

\begin{definition}
    For $\alpha, \beta\in\ord$, we define $\alpha\cdot\beta$ and $\alpha^\beta$ by transfinite recursion on $\beta$, where
    \begin{align*}
        \alpha \cdot 0 &= 0\\
        \alpha \cdot (\beta+1) &= (\alpha\cdot\beta) + \alpha && \text{for all }\beta\\
        \alpha \cdot \beta &= \textstyle\lim_{\gamma < \beta} {(\alpha\cdot\gamma)} &&\text{for limit\footnotemark}\ \beta
    \end{align*}
    and
        \begin{align*}
        \alpha^0 &= 1\\
        \alpha^{\beta+1} &= \alpha^\beta \cdot \alpha && \text{for all }\beta\\
        \alpha^\beta &= \textstyle\lim_{\gamma < \beta} \alpha^\gamma &&\text{for limit }\beta.
    \end{align*}
\end{definition}

\footnotetext{By convention, we do not regard 0 as a limit ordinal.}

In relation to the higher arithmetic on the natural numbers from Section 1, we now consider extending the arithmetic on the ordinals in a similar fashion, ideally reconstructing the hyperoperations on the finite ordinals. Here, we notice that the above definitions use ``leftward'' iteration in the successor case, which also applies when $\beta > 0$ is finite. Therefore, if we continue this scheme directly, we are met with the same lack of identity elements (as we observed previously), and we do not follow the definition of the hyperoperations anyway.

Thus, as a different approach, we may instead begin with the hyperoperations and extend their domain to the transfinite, adding the limit case above to their definition and using ordinal addition as the first term. Indeed, the right identity elements of natural arithmetic are the same in ordinal arithmetic. Already at multiplication, however, these operations diverge from the usual arithmetic on the ordinals, and are actually rendered trivial when the second argument is infinite. Letting $*$ denote the ordinal version of $[-,-]_2$, we would have, for all $\alpha \in \ord$,
\begin{align*}
    \alpha * (\omega + 1) &= \alpha + (\alpha * \omega) && \text{by definition} \\
                          &= \alpha + (\alpha \cdot \omega) && \text{by the limit case} \\
                          &= \alpha \cdot \omega && \\
                          &= \alpha * \omega && \text{by the limit case,}
\end{align*}
where $\cdot$ is the standard ordinal multiplication. By induction, it can then be shown that $\alpha * \beta = \alpha * \omega$ for all $\beta\geq\omega$. Essentially, it is inappropriate to introduce the operands from the left, since ordinal arithmetic (here, addition) is not commutative in the transfinite case.

Taken together, it seems that neither arithmetic can directly build a bridge over to the other. Therefore, we turn to a different perspective.

\section{Synthesis}

Until now, the operations considered have been defined by recursion on the second argument, which provides the depth of the iteration (whether finite or transfinite). In this section, we instead express the second argument as a sum of smaller terms, whereover the operation \emph{distributes} from the left. For example, let $* = [-,-]_n$ be a hyperoperation for some $n \in \omega$, and let $a, b \in \omega$ be natural numbers. Writing $*' = [-,-]_{n+1}$, we have

\[a *' \underbrace{(1+\cdots+1)}_{b}\ =\ \underbrace{(a *' 1)\ *\ \cdots\ *\ (a *' 1)}_{b},\]

where $*$ associates to the right. (Recalling that $a *' 1 = a$ clarifies the above.) This expresses a weak form of distributivity: Since the hyperoperations are not in general associative (as witnessed by exponentiation), we require the sum to be as ``decomposed'' as possible. Indeed, the natural number 1 cannot be expressed as a sum of strictly smaller naturals.

Inspired by this observation, we now attempt to formulate a similar identity on the ordinal numbers, expressing a sequence of operations $\op--_n : \ord^2 \to \ord$ with $\op\alpha\beta_0 = \alpha + \beta$ as the initial term. As above, we require the second argument, now expressed as an ordinal sum, to be maximally decomposed. Relating to the terms of such a sequence, we recall the following notion.

\begin{definition}
    An ordinal $\gamma > 0$ is called \emph{additive principal} if, for all $\alpha, \beta \in \ord$ with $\alpha, \beta < \gamma$, we have $\alpha+\beta < \gamma$.
\end{definition}

These are precisely the ordinals that cannot be further decomposed in terms of addition.

\begin{theorem}
    An ordinal $\beta > 0$ is additive principal if and only if it is of the form $\beta = \omega^\alpha$ for some $\alpha \in \ord$.
\end{theorem}

\begin{proof}
    Classic \cite{Sie1965}.
\end{proof}

In particular, we see that $\omega^0 = 1$ is the only finite additive principal ordinal. For additive principal ordinals $\omega^\alpha, \omega^\beta$, we have $\omega^\alpha + \omega^\beta = \omega^\beta$ just when $\alpha < \beta$. Thus, we may reduce a sum of additive principal ordinals to one in weakly decreasing order, giving the following familiar construct.

\begin{definition}\label{cnf}
    Let $\beta$ be an ordinal. For a natural number $k$ and ordinals $\beta_1 \geq \cdots \geq \beta_k$, we write

    \[\beta \stackrel{\textsc{cnf}}{=} \omega^{\beta_1} + \cdots + \omega^{\beta_k}\]

    in case the two expressions are equal, and call the right-hand side a \emph{Cantor normal form} (CNF) of $\beta$.
\end{definition}

We pull on the central result about these.

\begin{theorem}
    Every ordinal has a unique Cantor normal form.
\end{theorem}

\begin{proof}
    Also classic \cite{Sie1965}.
\end{proof}

\begin{example}
    An additive principal ordinal $\beta = \omega^\alpha$ has $k = 1$ and $\beta_1 = \alpha$ for its Cantor normal form. A finite ordinal $\beta > 0$ has $k = \beta$ and $\beta_1, \ldots, \beta_k = 0$. The ordinal $\beta = 0$ has $k = 0$ terms in its CNF, corresponding to the empty sum.
\end{example}

Returning to our desired operations, we can thus formulate their distributive identities as follows. Let $* = \op--_n$ and $*' = \op--_{n+1}$ for some $n\in\omega$. For $0<k\in\omega$, $\beta_1, \ldots, \beta_k \in \ord$ and ordinals $\alpha, \beta$ such that $\beta \stackrel{\textsc{cnf}}{=} \omega^{\beta_1} + \cdots + \omega^{\beta_k}$, we would like to have

\[\alpha *' (\omega^{\beta_1} + \cdots + \omega^{\beta_k})\ =\ (\alpha *' \omega^{\beta_1}) * \cdots * (\alpha *' \omega^{\beta_k}),\]

with $*$ associating to the right. As a definition, this is most meaningful when $k \geq 2$, such that $\beta$ is indeed written in smaller terms. For $k = 1$ and $\beta_1 > 0$, we have that $\beta$ is a limit ordinal; here, we let $\alpha *' \beta$ be the limit of the $\alpha *' \gamma$ for $\gamma < \beta$, in case this limit exists. When $\beta_1 = 0$, however, this would be inappropriate, since $\beta = 1$ would not be a limit ordinal; here, we instead let $\alpha *' 1 = \alpha$ as expected. Finally, when $k = 0$, we have $\beta = 0$, and we let $\alpha *' 0$ be the right identity element of $*$.

To collect this into a recursive definition, we only isolate the first term of the CNF of $\beta$. For concision, we write
\begin{align*}
    \beta_H &= \omega^{\beta_1}, \\
    \beta_T &= \phantom{\omega^{\beta_1} +} \omega^{\beta_2} + \cdots + \omega^{\beta_k}
\end{align*}
for $k \geq 1$, and call these the \emph{head} and \emph{tail} of $\beta$, respectively.

\begin{definition}\label{def:fold}
    Let $* : \ord^2 \to \ord$ be an operation with right identity element $\varepsilon$. The \emph{fold} of $*$ is the operation $*'$ given as follows. For $\alpha, \beta \in \ord$, the value $\alpha *' \beta$ is defined by transfinite recursion on $\beta$, where
    \begin{align*}
        \alpha *' 0 &= \varepsilon\\
        \alpha *' (1+\beta_T) &= \alpha * (\alpha *' \beta_T) &&\text{for }\beta > 0\text{ with }\beta_H = 1\\
        \alpha *' (\beta_H + \beta_T) &= \lim_{\gamma<\beta_H} (\alpha *' \gamma) * (\alpha *' \beta_T) && \text{for }\beta > 0\text{ with }\beta_H > 1.
    \end{align*}
\end{definition}

We then define $\op--_n : \ord^2 \to \ord$ by recursion on $n \in \omega$, where $\op--_0$ is ordinal addition and $\op--_{n+1}$ is the fold of $\op--_n$ for $n \in \omega.$ While 0 is the unique right identity element of $\op--_0$, it can be shown that $1$ is the unique right identity element of $\op--_n$ for all $0 < n \in \omega$.

\begin{remark}
    There is an edge case where Definition~\ref{def:fold} does not provide a value for $\op\alpha\beta_n$, namely, when $n \geq 3$, $\alpha = 0$, and $\beta \geq \omega$. This is because, for $n \geq 3$, $\op0\beta_n$ alternates between 1 and 0 as $\beta < \omega$ increases, so that the resulting $\omega$-sequence does not have a limit. In order to obtain a total operation, we find that setting $\op0\beta_n = 0$ for all additive principal $\beta > 1$ gives the most natural continuation of the course-of-values of $\op0-_n$ for $n \geq 3$; this amounts to amending Definition~\ref{def:fold} by taking limit infima instead of limits.
\end{remark}

Within this sequence, we now locate the hyperoperations of Section 2 and the usual ordinal arithmetic of Section 3, showing that it provides, at least on the surface, a common generalization of the two systems.

Firstly, focusing on the finite ordinals, we establish a correspondence between the hyperoperations $[-,-]_n$ and the above $\op--_n$.

\begin{theorem}
    For all $n\in\omega$, $a, b \in \omega$, we have $\langle a, b \rangle_n = [a, b]_n$.
\end{theorem}

\begin{proof}
    Let $a\in\omega$ be given. We prove the statement by simple induction on $n$.
    \begin{enumerate}
        \item[(i)] For $n=0$, it is immediate from Definition~\ref{hypop} and the definition of $\op--_0$.
        \item[(ii)] For some $n\in\omega$, suppose $\langle a, b \rangle_n = [a, b]_n$ for all $b\in\omega$. We show that $\op ab_{n+1} = [a, b]_{n+1}$ for all $b \in \omega$ by a certain form of induction on $b$.
        \begin{enumerate}
            \item[(1)] Let $b = 0$. For $n = 0$, we have $\op a0_1 = 0 = [a, 0]_1$. Otherwise, for $n \geq 1$, we have $n+1 \geq 2$, and thus $\op a0_{n+1} = 1 = [a, 0]_{n+1}$.
            \item[(2)] For $b = 1$, we have $\op a1_{n+1} = a = [a, 1]_{n+1}$.
            \item[(3)] Let $b\in\omega$ with $b\geq1$, and suppose $\op ab_{n+1} = [a, b]_{n+1}$. We then have $b+1 \geq 2$, and thus $b+1 = 1+b$ in terms of its head and tail. Therefore,
            \begin{align*}
                \op a{b+1}_{n+1} &= \op{\op a1_{n+1}}{\op ab_{n+1}}_n && \\
                &= \op a{\op ab_{n+1}}_n && \\
                &= [a, [a, b]_{n+1}]_n && \text{by both induction hypotheses} \\
                &= [a, b+1]_{n+1} && \text{by definition.}
            \end{align*}
        \end{enumerate}
    \end{enumerate}
\end{proof}

Secondly, focusing on the operations of ordinal arithmetic, we show that these correspond to the first three terms in our sequence $\op--_n$.

\begin{theorem}
    For all $\alpha, \beta \in \ord$, the following hold.
    \begin{enumerate}
        \item[\textnormal{(i)}] $\op\alpha\beta_0 = \alpha + \beta$
        \item[\textnormal{(ii)}] $\op\alpha\beta_1 = \alpha \cdot \beta$
        \item[\textnormal{(iii)}] $\op\alpha\beta_2 = \alpha^\beta$
    \end{enumerate}
\end{theorem}

\begin{proof}
    Let $\alpha \in \ord$ be given. We prove statements (i)--(iii) in order.
    \begin{enumerate}
        \item[(i)] Immediate.
        \item[(ii)] We proceed by a form of transfinite induction on $\beta$, splitting into cases according to Definition~\ref{def:fold}.
        \begin{enumerate}
            \item[(1)] For $\beta = 0$, we have $\op\alpha0_1 = 0 = \alpha\cdot0$.
            \item[(2)] Write $\beta = \beta_H + \beta_T > 0$ . For $\beta = 1$, we have $\op\alpha1_1 = \alpha = \alpha\cdot1$.
            \item[(3)] Let $\beta > 1$ be additive principal, and suppose $\op\alpha\gamma_1 = \alpha \cdot \gamma$ for all $\gamma < \beta$. We have
                \begin{align*}
                    \op\alpha\beta_1 &= \textstyle\lim_{\gamma < \beta} {\op\alpha\gamma_1} && \\
                    &= \textstyle\lim_{\gamma < \beta} {(\alpha \cdot \gamma)} && \text{by assumption} \\
                    &= \alpha \cdot \beta && \text{by definition.}
                \end{align*}
            \item[(4)] Write $\beta = \beta_H + \beta_T$. Suppose $\op\alpha\gamma_1 = \alpha \cdot \gamma$ for all $\gamma < \beta$. We have
                \begin{align*}
                    \op\alpha\beta_1 &= \op{\op\alpha{\beta_H}_1}{\op\alpha{\beta_T}_1}_0 && \\
                    &= \op\alpha{\beta_H}_1 + \op\alpha{\beta_T}_1 && \text{by (i)} \\
                    &= (\alpha \cdot \beta_H) + (\alpha \cdot \beta_T) && \text{since } \beta_H, \beta_T < \beta \\
                    &= \alpha \cdot (\beta_H + \beta_T) && \\
                    &= \alpha \cdot \beta. &&
                \end{align*}
        \end{enumerate}
        \item[(iii)] We proceed by transfinite induction on $\beta$.
        \begin{enumerate}
            \item[(1)] For $\beta = 0$, we have $\op\alpha0_2 = 1 = \alpha^0$ (in particular, $0^0 = 1$).
            \item[(2)] For $\beta = 1$, we have $\op\alpha1_2 = \alpha = \alpha^1$.
            \item[(3)] Let $\beta > 1$ be additive principal, and suppose $\op\alpha\gamma_2 = \alpha^\gamma$ for all $\gamma < \beta$. We have
                \begin{align*}
                    \op\alpha\beta_2 &= \textstyle\lim_{\gamma < \beta} {\op\alpha\gamma_2} && \\
                    &= \textstyle\lim_{\gamma < \beta} \alpha^\gamma && \text{by assumption} \\
                    &= \alpha^\beta && \text{by definition.}
                \end{align*}
            \item[(4)] Write $\beta = \beta_H + \beta_T$. Suppose $\op\alpha\gamma_2 = \alpha^\gamma$ for all $\gamma < \beta$. We have
                \begin{align*}
                    \op\alpha\beta_2 &= \op{\op\alpha{\beta_H}_2}{\op\alpha{\beta_T}_2}_1 && \\
                    &= \op\alpha{\beta_H}_2 \cdot \op\alpha{\beta_T}_2 && \text{by (ii)} \\
                    &= \alpha^{\beta_H} \cdot \alpha^{\beta_T} && \text{since } \beta_H, \beta_T < \beta \\
                    &= \alpha^{\beta_H + \beta_T} && \\
                    &= \alpha^\beta. &&
                \end{align*}
        \end{enumerate}
    \end{enumerate}
\end{proof}

In essence, the distributive definition of our operations manifests itself as special cases of the identities $\alpha \cdot (\beta + \gamma) = (\alpha\cdot\beta) + (\alpha\cdot\gamma)$ and $\alpha^{\beta+\gamma} = \alpha^\beta \cdot \alpha^\gamma$. Through this formulation, we can ensure that iteration is right-associative (as with the hyperoperations) while still being meaningful in the transfinite case (as with ordinal arithmetic).

\printbibliography

\end{document}